\documentclass{amsart}
\usepackage{amssymb,amsmath,latexsym,enumerate, dsfont,tikz,caption,subcaption,thm-restate,hyperref,cleveref}
\usepackage{graphicx,verbatim,enumerate,%
ifpdf,mdwlist}
\usepackage{xcolor}
\usepackage{tcolorbox}
\tcbuselibrary{breakable}
\usetikzlibrary{arrows,calc,shapes,decorations.pathreplacing,positioning}
\usepackage{color}
\usepackage{mathabx}
\ifpdf
\usepackage{hyperref}
\else
\usepackage[hypertex]{hyperref}
\fi

\newtheorem{theorem}{Theorem}[section]

\newtheorem{lemma}[theorem]{Lemma}

\newtheorem{proposition}[theorem]{Proposition}

\newtheorem{definition}[theorem]{Definition}

\newtheorem{exa}[theorem]{Example}
\newenvironment{example}{\begin{exa} \rm}{ \end{exa}}
\newenvironment{remark}{\begin{rem} \rm}{ \end{rem}}

\newtheorem{rem}[theorem]{Remark}

\title[Intrinsic complexity of recursive functions on $(\omega,<)$]{Intrinsic complexity of recursive functions on natural numbers with standard order}
\author[N.~Bazhenov]{Nikolay Bazhenov}
\address{Sobolev Institute of Mathematics, pr. Akad.
Koptyuga 4, Novosibirsk, 630090 Russia
}
\email{bazhenov@math.nsc.ru}

\author[D.~Kalociński]{Dariusz Kalociński}
\address{Institute of Computer Science, Polish Academy of Sciences, ul. Jana Kazimierza 5, 01-248 Warsaw, Poland}\email{dariusz.kalocinski@ipipan.waw.pl}

\author[M.~Wrocławski]{Michał Wrocławski}
\address{Department of Philosophy, University of Warsaw,
ul. Krakowskie Przedmieście 3, 00-927 Warsaw, Poland}
\email{m.wroclawski@uw.edu.pl}

\thanks{Kalociński was supported by the National Science Centre grant no. 2018/31/B/HS1/04018.}



\begin{document}

\maketitle

\begin{abstract}
Intrinsic complexity of a relation on a given computable structure is captured by the notion of its degree spectrum---the set of Turing degrees of images of the relation in all computable isomorphic copies of that structure. We investigate the intrinsic complexity of unary total recursive functions on nonnegative integers with standard order. According to existing results, possible spectra of such functions include three sets consisting of precisely: the computable degree, all c.e. degrees and all $\Delta_2$ degrees. These results, however, fall far short of the full classification. In this paper, we obtain a more complete picture by giving a few criteria for a function to have intrinsic complexity equal to one of the three candidate sets of degrees. Our investigations are based on the notion of block functions and a broader class of quasi-block functions beyond which all functions of interest have intrinsic complexity equal to the c.e. degrees. We also answer the questions raised by Wright \cite{wright_degrees_2018} and Harrison-Trainor \cite{harrison-trainor_degree_2018} by showing that the division between computable, c.e. and $\Delta_2$ degrees is insufficient in this context as there is a unary total recursive function whose spectrum contains all c.e. degrees but is strictly contained in the $\Delta_2$ degrees.
\end{abstract}

\section{Introduction}\label{sec:introduction}
In mathematics we study structures of various sorts like rings, fields or linear orders. In computability theory we investigate complexity of countable objects. A combination of the two---computable structure theory---examines the relationship between complexity and structure in the above sense \cite{ash_computable_2000, montalban_computable_2021}. One of the main research programs in computable structure theory consists in the study of how complexity of a relation on a given structure behaves under isomorphisms (see, e.g., \cite{richter_degrees_1981,harizanov_degree_1987,hirschfeldt_degree_2000,fokina_comp_model_20l4}). Recall that a structure is computable if its domain and basic relations are uniformly computable. Complexity of a relation might be captured by a measure such as Turing degrees. This leads to the notion of the degree spectrum (of a computable relation on a computable structure)---the set of Turing degrees assumed by the images of that relation in all computable isomorphic copies of that structure. This notion captures what might otherwise be called the intrinsic complexity of a relation. 

A natural motivation for investigating intrinsic complexity comes from treating computable copies of a structure as notations: we regard the elements of the copy as names for the members of the structure, with the underlying isomorphism acting as a naming function. A computable copy of a structure is thus a notation in which all the basic relations are computable (meaning that their images within the copy are computable). This is essentially Shapiro's idea, as studied, though in a very restricted sense, in \cite{shapiro_acceptable_1982}. But this analogy goes further. Shapiro insisted, not without reason, that computations are not performed directly on numbers but rather on their names (using the terminology of computable structure theory: computations are not performed on the underlying structure but on isomorphic copies). This intuition transfers to all computation-dependent notions, including complexity. In the end, the intricate notion of intrinsic complexity boils down to the study of how difficult it is to compute the relation in notations in which all the basic relations are computable.

Following Downey et al. \cite{downey_degree_2009} and Wright \cite{wright_degrees_2018}, we investigate degree spectra on the most common ordering: non-negative integers with the standard \emph{less than} relation, denoted by $(\omega,<)$. We study this question in the restricted setting of specific binary relations of general interest---graphs of unary total computable functions. As an example of how isomorphism might influence the complexity of a such a function, consider the successor. By a well-known result (see, e.g., Example 1.3 in \cite{chubb_degree_2009}), there is an isomorphic copy of $(\omega,<)$ in which the image of the successor computes the halting problem.

Several results from the literature partially characterize degree spectra of such functions. Moses \cite{moses_relations_1986} provided a syntactical characterization of intrinsically computable (i.e. having only the computable degree in their spectrum) $n$-ary relations on $(\omega,<)$. These results imply that a total unary recursive function is intrinsically computable if and only if it is almost constant or almost identity (see, Proposition \ref{theorem:trivial}). In \cite{downey_degree_2009}, Downey, Khoussainov, Miller and Yu examined degree spectra of unary relations on $(\omega,<)$. Their results show, among others, that the spectrum of any infinite coinfinite computable unary relation on $(\omega,<)$ contains all c.e. degrees (Theorem 1.1 in~\cite{downey_degree_2009}).
Wright extended their results by showing the following.

\begin{theorem}[Wright \cite{wright_degrees_2018}]
The spectrum of a computable $n$-ary relation which is not intrinsically computable contains all c.e. degrees.\label{thm:wright}
\end{theorem}
He was also able to show that a computable \emph{unary} relation which is not intrinsically computable has $\Delta_2$ degrees as a spectrum (see, also, \cite{knoll_degree_2009}).

Wright asked in \cite{wright_degrees_2018} whether the computable, the c.e. and the $\Delta_2$ degrees exhaust possible degree spectra for computable $n$-ary relations on $(\omega, <)$. Roughly at about the same time, Harrison-Trainor posed a related question in \cite{harrison-trainor_degree_2018} where he showed that either
\begin{enumerate}
    \item[(1)] there is a computable relation $R$ on $(\omega, <)$ whose degree spectrum strictly contains the c.e. degrees but does not contain all of the $\Delta_2$ degrees, or\label{HT_case1}
    \item[(2)] there is a computable relation $R$ on $(\omega, <)$ whose degree spectrum is all of the $\Delta_2$ degrees but does not have this degree spectrum uniformly.
\end{enumerate}
Harrison-Trainor conjectured that (\ref{HT_case1}) holds for the relation he constructed. We construct a unary total computable function (hence, a computable binary relation) witnessing (\ref{HT_case1}). This also answers Wright's question.

Results of this paper are heavily based on certain structural characteristics of functions, which we refer to as the block and (a weaker) quasi-block property. Intuitively, each block function on $(\omega,<)$ is defined by multiple sub-functions where each sub-function applies to a different finite $<$-interval of $\omega$ (Definition \ref{def:block_function}). A quasi-block function is one for which there are increasingly long initial $<$-segments such that no number from within the segment is sent outside it. The usefulness of these properties is clear in view of the observation that any computable non-quasi-block function has exactly all c.e. degrees as a spectrum (Theorem \ref{theorem:non_quasi_block}). One of the main contributions of the paper consists in the complete characterization of degree spectra of block functions which have at most finitely many isomorphism types of their elementary sub-functions (Theorem \ref{theorem:block-functions}). The second main contribution is Theorem \ref{theorem:unusual_spectrum} which answers Wright's and Harrison-Trainor's questions.

\section{Definitions}
\label{sec:definitions}
\begin{definition}
$(\omega, \prec)$ is a computable copy of $(\omega,<)$ if $\prec$ is a computable ordering on $\omega$ and $(\omega, <)$ and $(\omega, \prec)$ are isomorphic.
\end{definition}
\begin{definition}
Let $R$ be a relation on $(\omega,<)$, i.e. $R \subseteq \omega^k$, for some $k \in \omega$, and let $\mathcal A$ be a computable copy of $(\omega,<)$. If $\varphi$ is an isomorphism from $(\omega,<)$ to $\mathcal A$, we write $R_\mathcal{A}$ for the image of $R$ under $\varphi$. 
\end{definition}
\begin{definition}
Let $R$ be a relation on a computable copy of $(\omega,<)$. The degree spectrum or spectrum of $R$ on $(\omega,<)$, in symbols $DgSp_{(\omega,<)}(R)$, is the set of Turing degrees of $R_\mathcal{A}$ over all computable copies $\mathcal A$ of $(\omega,<)$. 
\end{definition}
Throughout the article, we use abbreviated forms: \emph{spectrum of $R$} and $DgSp(R)$.

\begin{definition}
Let $R$ be a relation on $(\omega,<)$. The relation $R$ is intrinsically computable if $DgSp(R)$ contains only the computable degree.
\end{definition}

Let $\mathcal{A} = (A, <_{\mathcal{A}})$ be a linear order. If $a\leq_{\mathcal{A}} b$, then $[a;b]_{\mathcal{A}}$ and $[a;b)_\mathcal{A}$ denote the intervals $\{x\,\colon a\leq_{\mathcal{A}} x \leq_{\mathcal{A}} b \}$ and $\{x\,\colon a\leq_{\mathcal{A}} x <_{\mathcal{A}} b \}$, respectively. If the order $\mathcal{A}$ is clear from the context, then we omit the subscript $\mathcal{A}$. $Succ$ is the successor function on $(\omega, <)$. $\langle  \cdot, \cdot \rangle$ is the pairing function.
Computability-related notation is standard and follows \cite{soare_recursively_1987}. For example, $\leq_T$ denotes the Turing reduction.

If $X \subseteq \omega$ is a $\Delta_2$ set, then one can choose its \emph{computable approximation} $\xi(k,s)$, i.e. a $\{0,1\}$-valued computable function such that $\lim_s \xi(k,s) = X(k)$, for all $k$. We often use notation $X_s(k)$ for $\xi(k,s)$.


\section{Results}\label{sec:results}
The following two statements will be useful.
\begin{restatable}{proposition}{thmTrivial}\label{theorem:trivial} Let $f$ be a unary total computable function. Then $f$ is intrinsically computable if and only if either $f$ is almost constant, or $f$ is almost identity.
\end{restatable}
\begin{proof}
    The result of Moses (Theorem~2 in~\cite{moses_relations_1986}) implies that a function $f$ is intrinsically computable if and only if there is a finite tuple $\bar a$ from $\omega$ and a quantifier-free formula $\theta(x,y,\bar{a})$ such that 
    \begin{equation}\label{equ:Moses}
        f(x) = y \ \Leftrightarrow\ (\omega,<) \models \theta(x,y,\bar{a}).
    \end{equation}
    If $f$ is almost constant or almost identity, then one can easily find $\bar{a}$ and $\theta(x,y,\bar{a})$ satisfying equation (\ref{equ:Moses}).
    
    Now suppose that $f$ is intrinsically computable, and $f$ is not equal to almost identity. Choose a tuple $\bar a = a_0 < a_1 < a_2 < \ldots < a_n$ and a formula $\theta$ satisfying (\ref{equ:Moses}). Choose a number $x_0 > a_n$ such that $f(x_0) \neq x_0$.
    
    We argue that $f(x_0) < x_0$. Indeed, assume that $f(x_0) > x_0$. Then, since $\theta$ is quan\-ti\-fi\-er-free and $x_0 > a_n$, it is clear that
    \[
        (\omega,<) \models \forall y[ y> x_0 \rightarrow \theta(x_0,y,\bar{a})].
    \]
    This contradicts the fact that $f$ is a function.   Now we have that $f(x_0) < x_0$ and $\theta(x_0,f(x_0),\bar{a})$ is satisfied. We deduce that
    \[
        (\omega,<) \models \forall x[ x\geq x_0 \rightarrow \theta(x,f(x_0),\bar{a})].
    \]
    Thus, for almost all $x$, the value $f(x)$ equals $f(x_0)$, i.e., $f$ is almost constant.
\end{proof}

\begin{restatable}[see, e.g., Example 1.3 in \cite{chubb_degree_2009}]{proposition}{thmDgSpSucc}\label{theorem:DgSp_Succ}
The spectrum of successor is equal to the c.e. degrees.
\end{restatable}

\begin{theorem}\label{theo:finite-range}
Let $f$ be a unary computable function with finite range. If $f$ is not intrinsically computable then its spectrum is equal to the $\Delta_2$ degrees.
\end{theorem}
\begin{proof}
The proof is based on the ideas from Theorem~1.2 of~\cite{wright_degrees_2018}. We provide a detailed exposition, so that a reader could familiarize with the proof techniques.

We fix $c_0\neq c_1$ such that $f^{-1}(c_i)$ is infinite. Without loss of generality, one may assume that $c_0 = 0$ and $c_1 = 1$.

Let $X\subseteq \omega$ be an arbitrary $\Delta_2$ set. We build a computable isomorphic copy $\mathcal{A} = (\omega, <_{\mathcal{A}})$ of the order $(\omega,<)$ such that $f_{\mathcal{A}}$ is Turing equivalent to the set $X$.
Our construction will ensure that the following two conditions hold: 
\begin{enumerate}
    \item[(i)] $k\in X$ if and only if $f_{\mathcal{A}}(2k) = 1$, for all $k$;
    
    \item[(ii)] the restriction of $f_{\mathcal{A}}$ to the set of odd numbers (i.e., $f_{\mathcal{A}}\upharpoonright \{2k+1\,\colon k\in\omega \}$) is computable.
\end{enumerate}
It is clear that these conditions imply $f_{\mathcal{A}}\equiv_T X$.

Let $M$ be a large enough natural number such that
\[
    (\forall x > M) [\text{the $f$-preimage of } f(x) \text{ is infinite, and } x\not\in \mathrm{range}(f)].
\]
Beforehand, we use odd numbers to copy the initial segment $[0;M]$ of $(\omega, <)$. More formally, we put $2k+1 <_{\mathcal{A}} 2l+1$ for all $k < l \leq M$. In addition, any newly added (to the copy $\mathcal{A}$) number will be strictly $\mathcal{A}$-greater than $2M+1$.

Our construction satisfies the following requirements:
\[
   \begin{array}{l} 
        e\in X \ \Leftrightarrow\ f_{\mathcal{A}}(2e) = 1,\\
        e\not\in X \ \Leftrightarrow\ f_{\mathcal{A}}(2e) = 0.
   \end{array}
   \tag{$\mathcal{R}_e$}
\]
As usual, this will be achieved by working with a computable  approximation $X_s(e)$.

By $\mathcal{A}_s$ we denote the finite structure built at a stage $s$. At each stage $s$, there is a natural isomorphic embedding $h_s$ from $\mathcal{A}_s$ into $(\omega,<)$. If $\mathcal{A}_s$ consists of $a_0 <_{\mathcal{A}} a_1 <_{\mathcal{A}} a_2 <_{\mathcal{A}} \ldots <_{\mathcal{A}} a_n$, then we assume that $h_s(a_i) = i$, for all $i\leq n$.

This convention allows one to talk about values $f_{\mathcal{A}_s}(x)$ for elements $x\in \mathcal{A}_s$. We simply assume that
\[
    f_{\mathcal{A}_s}(a_i) = h_s^{-1} \circ f \circ h_s(a_i).
\]
Our construction will ensure that $f_{\mathcal{A}}(x) = \lim_s f_{\mathcal{A}_s}(x)$, for all $x$. Sometimes (when the usage context is unambiguous), we write just $f_{\mathcal{A}}(x)$ in place of $f_{\mathcal{A}_s}(x)$.


\emph{Strategy $\mathcal{R}_e$ in isolation.} Suppose that $(s_0+1)$ is the first stage of work for this strategy. Then we add $2e$ to the right end of $\mathcal{A}$. Since we want to ensure that $f_{\mathcal{A}_{s_0+1}}(2e) = X_{s_0+1}(e)$, we also add (if needed) finitely many fresh odd numbers in-between $\mathcal{A}_s$ and $2e$, i.e., we set
\[
    a <_{\mathcal{A}} 2k+1 <_{\mathcal{A}} 2e, 
\]
for $a\in \mathcal{A}_{s_0}$ and newly added numbers $2k+1$.

We say that $\mathcal{R}_e$ \emph{requires attention} at a stage $s$ if the current value $f_{\mathcal{A}_s}(2e)$ is not equal to $X_s(e)$. In order to deal with $\mathcal{R}_e$, we introduce the following important ingredient of our proof techniques. For the sake of future convenience, we give a \emph{general} description of the module.

\begin{tcolorbox}[breakable,title={\textbf{Pushing-to-the-right module (PtR-module).}},boxrule=0pt,colback=gray!10,colframe=gray!10,coltitle=black]
We split the (current finite) structure $\mathcal{A}_s$ into three intervals:
$B <_{\mathcal{A}} C <_{\mathcal{A}} D$, where, say, we have $B = [a;b]_{\mathcal{A}}$, $C = \{ c^0 <_{\mathcal{A}} c^1 <_{\mathcal{A}} \ldots <_{\mathcal{A}} c^m\}$, and $D = \{d^0 <_{\mathcal{A}} d^1 <_{\mathcal{A}} \ldots <_{\mathcal{A}} d^n\}$. Informally speaking, the module aims to achieve the following goal: while preserving all values $f_{\mathcal{A}}(x)$ for $x\in B \cup D$, we want to change the function $f_{\mathcal{A}} \upharpoonright C$ in such a way that $f_{\mathcal{A}}$ satisfies a particular requirement. In addition, we require that $C$ remains an interval inside $\mathcal{F}$.

More formally, we extend the structure $\mathcal{A}_s$ to a finite structure $\mathcal{F}$ (which is intended to be an initial segment of $\mathcal{A}_{s+1}$) with the following properties:
\begin{itemize}
    \item every element $x\in \mathcal{F} \setminus \mathcal{A}_s$ is a fresh odd number, and each such $x$ satisfies either $B <_{\mathcal{A}} x <_{\mathcal{A}} C$ or $x >_{\mathcal{A}} C$;

    \item $f_{\mathcal{F}}(d^i) =  f_{\mathcal{A}_s}(d^i)$ for all $i\leq n$; 
    
    \item the new values $f_{\mathcal{F}}(c^j)$ satisfy some \emph{target condition}.
\end{itemize}
In the future, when we talk about a particular instance of the module, we will always explicitly specify the desired target condition.

Roughly speaking, our module keeps the interval $B$ fixed, while all elements from $C \cup D$ are pushed to the right (with the help of newly added odd numbers). In addition, the elements of $C$ stick together.
\end{tcolorbox}

Going back to $\mathcal{R}_e$: if $\mathcal{R}_e$ requires attention at a stage $s$, then we implement the following actions.

\emph{The PtR-module for the strategy $\mathcal{R}_e$.} In our $\mathcal{R}_e$-setting, we choose the middle interval $C$ as the singleton $\{ 2e\}$. The desired target condition is a natural one: we aim to satisfy $f_{\mathcal{A}}(2e) = X_s(e)$.

We build a finite structure $\mathcal{F}$ extending $\mathcal{A}_s$ as dictated by the PtR-module. Then we declare that $\mathcal{F}$ is the output of our module, and proceed further. This concludes the description of the $\mathcal{R}_e$-strategy.

\smallskip

\emph{Construction.} At a stage $s+1$, we work with strategies $\mathcal{R}_e$, for $e\leq s$. So, a strategy $\mathcal{R}_e$ starts working at the stage $e+1$. For each $\mathcal{R}_e$ (in turn), our actions follow the description given above. After $\mathcal{R}_i$ finished its work, the PtR-module of the next strategy $\mathcal{R}_{i+1}$ works with the finite structure produced by $\mathcal{R}_i$. 
Since the described PtR-module preserves $f_{\mathcal{A}} \upharpoonright (B \cup D)$, our strategies do not injure each other. We define $\mathcal{A} = \bigcup_{s\in\omega} \mathcal{A}_s$, where $\mathcal{A}_{s+1}$ is the final content of our structure produced by the PtR-module of $\mathcal{R}_s$ at the end of stage $s+1$.

\smallskip

\emph{Verification.} First, we show that in the construction, every application of a PtR-module is successful (i.e., one can always build a desired structure $\mathcal{F}$).

In order to prove this, we consider our structures from a different angle: The structure $(\omega,<,f)$ can be treated as an infinite string $\beta$ over a finite alphabet $\Sigma = \mathrm{range}(f)$, where the $i$-th symbol $\beta(i)$ of the string is equal to $f(i)$, $i\in \omega$. 

Then the construction of $\mathcal{F}$ in the PtR-module can be re-interpreted as follows. We are given three finite strings: $\sigma$, $\tau$ (of length one), and $\rho$ (of length $n+1$)~--- for the intervals $B$, $C = \{ 2e\}$, and $D$ correspondingly. Our task is to find finite strings $\tau',\rho'_0,\rho'_1,\ldots,\rho'_n$ with the following property:
\[
    \sigma \, \tau' \, a \, \rho'_0 \, \rho(0) \, \rho'_1 \, \rho(1) \, \ldots \, \rho'_n \, \rho(n),
\]
where $a = X_s(e)$, is an initial segment of $\beta$.

This task can be always implemented successfully~--- this is a consequence of the following simple combinatorial fact. 

\begin{remark}\label{rem:inf-string-01}
Let $\Sigma$ be a finite alphabet, and let $\alpha \in \Sigma^{\omega}$ be an infinite string over $\Sigma$. Suppose that every symbol from $\Sigma$ occurs infinitely often in $\alpha$. Then for every finite string $\sigma \in \Sigma^{<\omega}$ of length $m>0$, one can find finite strings $\tau_0,\tau_1,\ldots,\tau_{m-1}$ such that
\[
    \tau_0 \, \sigma(0) \, \tau_1 \, \sigma(1) \, \ldots \, \tau_{m-1} \, \sigma(m-1) \text{ is an initial segment of } \alpha.
\]
\end{remark}

So, we deduce that all applications of a PtR-module are successful. Hence, if $e\leq s$, then by the end of the stage $s+1$ we have $f_{\mathcal{A}_{s+1}}(2e) = X_s(e)$. This implies that every requirement $\mathcal{R}_e$ is satisfied. 

Each element $a\in\mathcal{A}$ moves (to the right) only finitely often. Indeed, there are only finitely many even numbers $2e$ such that $2e \leq_{\mathcal{A}} a$. Consider a stage $s^{\ast}$ such that the values $X_s(e)$ (for these $2e$) never change after $s^{\ast}$. Clearly, the element $a$ never moves after the stage $s^{\ast}$. 

We deduce that the structure $\mathcal{A}$ is a computable copy of $(\omega, <)$. For every $k$, after the value $f_{\mathcal{A}_s}(2k+1)$ is defined for the first time, this value never changes (since the PtR-mo\-dule always preserves the restriction $f_{\mathcal{A}} \upharpoonright (B \cup D)$). Therefore, our structure $\mathcal{A}$ satisfies Conditions~(i) and~(ii) defined above. Theorem~\ref{theo:finite-range} is proved.
\end{proof}

\subsection{Block functions}
From now on, we study some natural subclasses of unary total recursive functions with infinite range.
\begin{definition}\label{def:block_function}
    We say that a total function $f\colon \omega \to \omega$ is a \emph{block function} if for every $a \in \omega$, there is a finite interval $I$ of $(\omega,<)$ with the following properties:
    \begin{itemize}
        \item $a \in I$;
        
        \item $I$ is closed under $f$ (i.e, for all $x\in I$, $f(x) \in I$);
        
        \item $I$ is closed under $f^{-1}$ (i.e., for all $x\in I$, $f^{-1}(x) \subseteq I$).
    \end{itemize}
    Suppose that $I$ is the least such interval (with respect to set-theoretic inclusion). Then we say that the structure $(I, <, f\upharpoonright I)$ is an \emph{$f$-block} (of the element $a$). 
    
    If $(I, <, f\upharpoonright I)$ is an $f$-block, we refer to its isomorphism type as an $f$-type (or a type).
\end{definition}
\begin{remark}\label{representation}
For any computable block function $f$ there is a 1-1 computable enumeration of its types. $f$ can be represented by the infinite string $\alpha_f \in [0;N)^\omega$, where $[0,N)$ is the domain of the enumeration, for some $N \in \omega \cup \{+\infty\}$. For example, if $I_0,I_1,\ldots,I_N$ are all (isomorphism types of) $f$-blocks, then $(\omega,<,f)$ can be treated as an infinite string $\alpha_f \in \{n: 0 \leq n \leq N\}^{\omega}$, e.g. a string $012012012 \ldots$ corresponds to a disjoint sum of the following form:
$I_0 + I_1 + I_2 + I_0 + I_1 + I_2 + I_0 + I_1 + I_2 + \ldots $
\end{remark}
\begin{example} $f(n) = 2 \cdot \lfloor \frac{n}{2} \rfloor$ is a block function. Its spectrum consists of all $\Delta_2$ degrees by Theorem \ref{theorem:block-functions} below.
\end{example}

\begin{example}\label{ex:involution}
Consider finite structures $\mathcal J_n$ from Figure \ref{fig:involution}. Let $g$ be the involution such that $(\omega, < ,g) \cong \mathcal J_0 + \mathcal J_1 + \mathcal J_2 + \ldots$ Clearly, $g$ is a block function. Proposition \ref{involution_proof} shows its degree spectrum is all of the c.e. degrees.
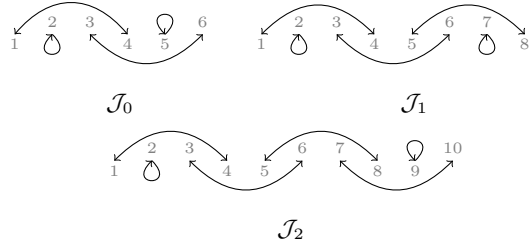
\begin{figure}
\captionsetup[subfigure]{labelformat=empty}
\begin{subfigure}[b]{0.25\textwidth}
\begin{tikzpicture}
		\node [label={[gray,yshift=-0.45cm]\tiny{$1$}}] (0) at (-6.5-10, 3.75) {};
		\node [label={[gray,yshift=-0.15cm]\tiny{$2$}}] (1) at (-6-10, 3.75) {};
		\node [label={[gray,yshift=-0.15cm]\tiny{$3$}}] (2) at (-5.5-10, 3.75) {};
		\node [label={[gray,yshift=-0.45cm]\tiny{$4$}}] (3) at (-5-10, 3.75) {};
		\node [label={[gray,yshift=-0.45cm]\tiny{$5$}}] (4) at (-4.5-10, 3.75) {};
		\node [label={[gray,yshift=-0.15cm]\tiny{$6$}}] (5) at (-4-10, 3.75) {};
		\draw [<->,bend left=50, looseness=1.25] (0.center) to (3.center);
		\draw [->,in=-135, out=-45, loop] (1.center) to ();
		\draw [<->,bend right=45, looseness=1.25] (2.center) to (5.center);
		\draw [->,in=135, out=45, loop] (4.center) to ();
\end{tikzpicture}
\caption{$\mathcal J_0$}
\end{subfigure}
\begin{subfigure}[b]{0.35\textwidth}
\begin{tikzpicture}
		\node [label={[gray,yshift=-0.45cm]\tiny{$1$}}] (0a) at (-11.5-1.5, 3.75) {};
		\node [label={[gray,yshift=-0.15cm]\tiny{$2$}}] (1a) at (-11-1.5, 3.75) {};
		\node [label={[gray,yshift=-0.15cm]\tiny{$3$}}] (2a) at (-10.5-1.5, 3.75) {};
		\node [label={[gray,yshift=-0.45cm]\tiny{$4$}}] (3a) at (-10-1.5, 3.75) {};
		\node [label={[gray,yshift=-0.45cm]\tiny{$5$}}] (4a) at (-9.5-1.5, 3.75) {};
		\node [label={[gray,yshift=-0.15cm]\tiny{$6$}}] (5a) at (-9-1.5, 3.75) {};
		\node [label={[gray,yshift=-0.15cm]\tiny{$7$}}] (6a) at (-8.5-1.5, 3.75) {};
		\node [label={[gray,yshift=-0.45cm]\tiny{$8$}}] (7a) at (-8-1.5, 3.75) {};
		\draw [<->,bend left=45, looseness=1.25] (0a.center) to (3a.center);
		\draw [->,in=-135, out=-45, loop] (1a.center) to ();
		\draw [<->,bend right=45, looseness=1.25] (2a.center) to (5a.center);
		\draw [->,in=-135, out=-45, loop] (6a.center) to ();
        \draw [<->,bend left=45, looseness=1.25] (4a.center) to (7a.center);
\end{tikzpicture}
\caption{$\mathcal J_1$}
\label{}
\end{subfigure}
\begin{subfigure}[b]{0.4\textwidth}
\begin{tikzpicture}
        \node [label={[gray,yshift=-0.45cm]\tiny{$1$}}] (0b) at (-11.5+3, 3.75) {};
		\node [label={[gray,yshift=-0.15cm]\tiny{$2$}}] (1b) at (-11+3, 3.75) {};
		\node [label={[gray,yshift=-0.15cm]\tiny{$3$}}] (2b) at (-10.5+3, 3.75) {};
		\node [label={[gray,yshift=-0.45cm]\tiny{$4$}}] (3b) at (-10+3, 3.75) {};
		\node [label={[gray,yshift=-0.45cm]\tiny{$5$}}] (4b) at (-9.5+3, 3.75) {};
		\node [label={[gray,yshift=-0.15cm]\tiny{$6$}}] (5b) at (-9+3, 3.75) {};
		\node [label={[gray,yshift=-0.15cm]\tiny{$7$}}] (6b) at (-8.5+3, 3.75) {};
		\node [label={[gray,yshift=-0.45cm]\tiny{$8$}}] (7b) at (-8+3, 3.75) {};
		\node [label={[gray,yshift=-0.45cm]\tiny{$9$}}] (8b) at (-7.5+3, 3.75) {};
		\node [label={[gray,yshift=-0.15cm]\tiny{$10$}}] (9b) at (-7+3, 3.75) {};
		\draw [<->,bend left=45, looseness=1.25] (0b.center) to (3b.center);
		\draw [->,in=-135, out=-45, loop] (1b.center) to ();
		\draw [<->,bend right=45, looseness=1.25] (2b.center) to (5b.center);
		\draw [->,in=135, out=45, loop] (8b.center) to ();
        \draw [<->,bend left=45, looseness=1.25] (4b.center) to (7b.center);
        \draw [<->,bend right=45, looseness=1.25] (6b.center) to (9b.center);
\end{tikzpicture}
\caption{$\mathcal J_2$}
\end{subfigure}
\caption{Structures $\mathcal J_n = ([1;6+2n],<,f)$, for $n=0,1,2$, where $f$ is the involution such that $f(k) = k$ iff $k=2$ or $k = 6+2n-1$, and $f(k) = k+3$ for odd numbers $\leq 6 + 2n -3$.}
\label{fig:involution}
\end{figure}
\end{example}

\begin{proposition}\label{involution_proof}
The degree spectrum of $g$, as defined in Example \ref{ex:involution}, is all of the c.e. degrees. 
\end{proposition}
\begin{proof}
Observe that $\deg(c_g) = \mathbf{0}$. By Proposition \ref{thm:nikolay}, Theorem \ref{thm:wright} and Proposition \ref{theorem:trivial}, the spectrum of $g$ is all of the c.e. degrees. For the sake of illustration, below we give a separate proof.

We show that the RS-module introduced in Section \ref{sec:qbf} can be successfully applied to $(\omega,<,g)$ which implies that the spectrum of $g$ is all of the c.e. degrees.
\par
Let $\mathcal I$ be an initial segment of a computable copy $\mathcal A$ of $(\omega,<)$. Let $\mathcal R$ be the condition saying that $\mathcal I \cong \mathcal J_0 + \mathcal J_1 + \ldots + \mathcal J_k$, for some $k \in \omega$. We show that $\mathcal R$ is decidable in $g_\mathcal{A}$. 
\par
Observe that given any $a \in \mathcal I$, one can computably in $g_\mathcal{A}$ recreate the $g$-block of $a$ in $\mathcal A$. The idea is that, initially, the recreated fragment $G:= \{a\}$. We use the following procedure of one argument $x$, starting from $x = a$. 
\par
$(\star)$ We compute $g_\mathcal{A}(x)$ and, if $g_\mathcal{A}(x) = x$, then we know there is precisely one number $y$ such that $x$ is between $y_1 := y$ and $y_2:= g_\mathcal{A}(y)$, whereas if $g_\mathcal{A}(x) \neq x$, we know that there are precisely two numbers $y_1, y_2$ between $x$ and $g_\mathcal{A}(x)$. In either case, we find $y_1,y_2$ accordingly, and for each $y \in \{y_1,y_2\}$, if $y$ is not yet in $G$, we add it to $G$ and we run $(\star)$ on input $y$.
\par
One can prove that the above algorithm, when run with $a$ on input, stops and after stopping, $G$ consists of precisely the elements of the $g$-block of $a$ in $\mathcal A$.
\par
The above fact allows us to decide whether, for each $a \in \mathcal I$, $\mathcal I$ contains the whole $g$-block of $a$. This is sufficient, since $\mathcal I$ is an initial segment and, therefore, if it consists of full blocks only, then it must have the desired form.
\par
It remains to show that given any initial $<_\mathcal{A}$-segment $\mathcal I_t \cong \mathcal J_0 + \mathcal J_1 + \ldots + \mathcal J_t$, we can computably in $g_\mathcal{A}$ extend it to the $<_\mathcal{A}$-segment $\mathcal I_{t+1}\cong \mathcal J_0 + \mathcal J_1 + \ldots + \mathcal J_{t+1}$. We can proceed in stages as follows. At stage $0$ we declare that all numbers outside $\mathcal I_t$ are unused. At any stage $s$, we take the least unused number and, computably in $g_\mathcal{A}$, we recreate its $g$-block in $\mathcal A$. All numbers that are in this block join used elements. If the block is isomorphic to $\mathcal J_{t+1}$ (which we check computably in $g_\mathcal{A}$) then we already have $\mathcal I_{t+1}$ which becomes $\mathcal I_t$ extended by the $g$-block of $a$. Otherwise, we go to the next stage.
\par
Notice that in the above argument, we have used the fact that $g$-types are pairwise non-embeddable. Otherwise, we would not be sure, after finding some block isomorphic to $\mathcal J_{t+1}$ whether we actually have a copy of $\mathcal J_{t+1}$ or maybe a copy of some other type in which $\mathcal J_{t+1}$ is embedded.
\end{proof}

\begin{restatable}{theorem}{theoremOnBlockFunctions}\label{theorem:block-functions}
Let $f$ be a computable block function such that it has only finitely many $f$-types and $f$ is not almost identity.
Then the spectrum $DgSp(f)$ consists of all $\Delta_2$ degrees.
\end{restatable}
\begin{proof}
      Let $I_0,I_1,\ldots,I_N$ be all (isomorphism types of) $f$-blocks. We represent the structure $\mathcal B = (\omega,<,f)$ by $\alpha_f$ according to Remark \ref{representation}. As in Theorem~\ref{theo:finite-range}, we fix a $\Delta_2$ set $X$. 
    Our goal is to construct a computable copy $\mathcal{A} = (\omega,<_{\mathcal{A}})$ of $(\omega,<)$ such that $f_{\mathcal{A}} \equiv_T X$. In general, we follow the notations of Theorem~\ref{theo:finite-range} (e.g., $f_{\mathcal{A}_s}(x)$ is defined in the same way as in the previous proof).
    
    Beforehand, we choose a large enough number $M$ such that:
    \begin{itemize}
        \item $M$ lies at the right end of its $f$-block (inside $\mathcal{B}$),
        
        \item for every $x > M$, the isomorphism type of its $f$-block occurs infinitely often in $\mathcal{B}$.
    \end{itemize}
    As in the proof of Theorem~\ref{theo:finite-range}, we copy the interval $[0;M]$ into our structure $\mathcal{A}$, and all new elements will be added to the right of this interval.
    
    \smallskip
    
    The proof is split into three cases which depend on the properties of the string $\alpha_f$ (each of the cases requires a separate construction):
    \begin{enumerate}
        \item[(s)] There are two different finite strings $\sigma$ and $\tau$ such that:
        \begin{itemize}
            \item the lengths of $\sigma$ and $\tau$ are the same;
        
            \item $\tau$ can be obtained via a permutation of $\sigma$, i.e., there is a permutation $h$ of the set $\{0,1,\ldots,|\sigma|-1\}$ such that $\tau(i) = \sigma(h(i))$, for all $i<|\sigma|$;
        
            \item both $\sigma$ and $\tau$ occur infinitely often in $\alpha_f$.
        \end{itemize}
        
        \item[(b)] There is only one block $I_k$ such that $k$ occurs infinitely often in $\alpha_f$.
        
        \item[(c)] Neither of the previous two cases holds.
    \end{enumerate}
    
    \smallskip
    
    \textsc{Case~(a).} For the sake of simplicity, we give a detailed proof for the case when $\sigma = 01$ and $\tau = 10$. After that, we explain how to deal with the general case.
    
    
    Our construction satisfies the following requirements:
    \[
        \begin{array}{l} 
            e\in X \ \Leftrightarrow\ 2e \text{ belongs to a block isomorphic to } I_1,\\
        e\not\in X \ \Leftrightarrow\ 2e \text{ belongs to a block isomorphic to } I_0.
   \end{array}
   \tag{$\mathcal{R}_e$}
    \]
    Suppose that $|I_0| + |I_1| = q+1$.
    
    \emph{Strategy $\mathcal{R}_e$ in isolation.} When $\mathcal{R}_e$ starts working at a stage $s_0+1$, we proceed as follows. Assume that $X_{s_0}(e) = 1$ (the other case is treated similarly). We choose $q$ fresh odd numbers $c^e_1,c^e_2,\ldots,c^e_q$ and declare them the \emph{companions} of $2e$. We add the chain
    \[
        2e <_{\mathcal{A}} c^e_1 <_{\mathcal{A}} c^e_2 <_{\mathcal{A}} \ldots <_{\mathcal{A}} c^{e}_q
    \]
    to the right of $\mathcal{A}_{s_0}$. If needed, we add finitely many fresh odd numbers in-between $\mathcal{A}_{s_0}$ and $2e$. This procedure ensures that (at the moment) the finite structure $([2e;c^e_q]_{\mathcal{A}}, <_{\mathcal{A}}, f_{\mathcal{A}})$ is isomorphic to the disjoint sum $I_1 + I_0$.
    
    The strategy $\mathcal{R}_e$ \emph{requires attention} at a stage $s$ if inside the current $\mathcal{A}_s$, the number $2e$ belongs to a copy of $I_{1-X_s(e)}$. When $\mathcal{R}_e$ requires attention, we apply a PtR-module.
    
    \emph{The PtR-module for $\mathcal{R}_e$.} We choose the middle interval $C$ as the set containing $2e$ and all its companions, i.e. $C = \{2e <_{\mathcal{A}} c^e_1 <_{\mathcal{A}} \ldots <_{\mathcal{A}} c^e_q\}$. Our target condition is defined as follows: inside the resulting structure $\mathcal{F}$, the structure $(C, <_{\mathcal{F}}, f_{\mathcal{F}}\upharpoonright C)$ is isomorphic to the disjoint sum $I_{X_s(e)} + I_{1-X_s(e)}$.
    As in Theorem~\ref{theo:finite-range}, the structure $\mathcal{F}$ is treated as output of the module.
    
    
    \emph{The construction} is arranged similarly to that of  Theorem~\ref{theo:finite-range}.
    
    \emph{Verification.} We need to show that every application of a PtR-module is successful. This follows from two observations:
    \begin{enumerate}
        \item If we want to ``transform'', say, $I_0 + I_1$ into $I_1 + I_0$, then this can be achieved by an appropriate pushing to the right, since the string $\tau = 10$ occurs infinitely often in $\alpha_f$.
        
        \item  Remark~\ref{rem:inf-string-01} guarantees that one can also safely push the interval $D$ (from the PtR-module): notice that if some block $I_r$ occurs in $D$, then $r$ occurs  infinitely often in $\alpha_f$.
    \end{enumerate}
    Since pushing to the right is always successful, every requirement $\mathcal{R}_e$ is satisfied. Note that given $f_{\mathcal{A}}$ as an oracle, one can recover the $f_{\mathcal{A}}$-block of $2e$. This fact (together with $\mathcal{R}_e$-requirements) implies that $X \leq_T f_{\mathcal{A}}$.
    
    Every element $a\in\mathcal{A}$ is pushed to the right only finitely often. Therefore, the structure $\mathcal{A}$ is a computable copy of $(\omega, <)$.
    
    Given an odd number $x = 2k+1$, one can computably determine which of the following two cases holds:
    \begin{enumerate}
        \item $2k+1$ is a companion $c^e_t$ of some even number $2e$ (in this case, the indices $e$ and $t$ are also computed effectively), or
        
        \item $2k+1$ is added as a ``filler'' by some action of an $\mathcal{R}_e$-strategy (either by its initial actions, or by an application of a PtR-module).
    \end{enumerate}
    In the second case, the value $f_{\mathcal{A}_s}(x)$ never changes (after being defined for the first time). In the first case, the oracle $X$ can tell us whether $x = c^e_t$ belongs to (a copy of) $I_0$ or $I_1$, and $X$ can also compute the image $f_{\mathcal{A}}(x)$. In a similar way, $X$ computes the images $f_{\mathcal{A}}(2e)$, for $e\in\omega$. Hence, we obtain that $f_{\mathcal{A}} \equiv_T X$. This concludes the case when $\sigma = 01$ and $\tau = 10$.
    
    \smallskip
    
        The case of arbitrary $\sigma$ and $\tau$ follows a similar proof outline. We illustrate this by considering $\sigma = 012301$ and $\tau = 013021$. Then our construction will switch between finite structures
        \[
            \mathcal{F}_{\sigma} = I_0 + I_1 + I_2 + I_3 + I_0 + I_1 \text{ and } \mathcal{F}_{\tau} = I_0 + I_1 + I_3 + I_0 + I_2 + I_1.
        \]
        Since both $\sigma$ and $\tau$ occur infinitely often in $\alpha_f$, an appropriate PtR-module can always ``transform'' $\mathcal{F}_{\sigma}$ into $\mathcal{F}_{\tau}$, and vice versa.
        
        During the construction, an even number $2e$ will always belong to the third block from the left inside $\mathcal{F}_{\square}$ (i.e., either $I_2$ in $\mathcal{F}_{\sigma}$, or $I_3$ in $\mathcal{F}_{\tau}$). 
        The third block is chosen because it corresponds to the first position, where $\sigma$ and $\tau$ differ.
        The rest of the corresponding copy of $\mathcal{F}_{\square}$ consists of companions of $2e$. In the final structure $\mathcal{A}$, we will achieve the following: if $e\in X$, then $2e$ lies in a copy of $I_2$; otherwise, $2e$ belongs to a copy of $I_3$. This concludes the discussion of Case~(a).
        
    \textsc{Case~(b).} 
    Without loss of generality, we assume that $I_k = I_0$. We satisfy the following requirements:
    \[
        \begin{array}{l} 
            e\in X \ \Leftrightarrow\ 2e \text{ lies at the right end of a copy of } I_0,\\
            e\not\in X \ \Leftrightarrow\ 2e \text{ lies at the left end of a copy of } I_0.
        \end{array}
        \tag{$\mathcal{R}_e$}
    \]
    Suppose that $|I_0| = q+1$. Notice that $q\geq 1$, since $f$ is not almost identity.
    
    \emph{Strategy $\mathcal{R}_e$ in isolation.}
    $2e$ will have finitely many odd numbers as its \emph{companions}. In contrast to Case~(a), these companions could be added stage-by-stage.
    
    When $\mathcal{R}_e$ starts working at a stage $s_0+1$, we proceed as follows. Suppose $X_{s_0}(e) = 1$ (the other case is similar). Then we choose $q$ fresh odd numbers $c_1,\ldots,c_{q}$, and declare that they are companions of $2e$. We set $c_1 <_{\mathcal{A}} \ldots <_{\mathcal{A}} c_q <_{\mathcal{A}} 2e$ (these elements are added to the right of $\mathcal{A}_{s_0}$). We ensure that the structure $([c_1;2e]_{\mathcal{A}}, <_{\mathcal{A}}, f_{\mathcal{A}})$ is isomorphic to $I_0$ (if needed, one adds fresh odd numbers in-between $\mathcal{A}_{s_0}$ and $c_1$).
    
    We also ensure that by the end of each stage $s$, $2e$ and its (current) companions form an interval inside $\mathcal{A}_s$, and this interval can be treated as a sum of blocks (in $\mathcal{A}_s$).
    
    The strategy $\mathcal{R}_e$ \emph{requires attention} at a stage $s$ if inside the current $\mathcal{A}_s$, the corresponding requirement is not satisfied (e.g., if $X_s(e) =0$ and $2e$ lies at the right end of $I_0$). When $\mathcal{R}_e$ requires attention, we apply a PtR-module.
    
    \emph{The PtR-module for $\mathcal{R}_e$.} 
    We choose the middle interval $C$ as the set containing $2e$ and all its current companions. We consider the following two subcases.
    
    \emph{Subcase 1.} Assume that right now, $X_s(e) = 1$ and $2e$ lies at the left end of a copy of $I_0$. 
    Then our target condition is defined as follows: inside the resulting output structure $\mathcal{F}$, the number $2e$ should belong to the right end of a copy of $I_0$.
    
    In order to achieve this condition, we add precisely $q$ fresh odd numbers in-between $B$ and $C$, and only one fresh odd number in-between $C$ and $D$. This guarantees that $2e$ ``moves'' to the right end of a block.
    
    \emph{Subcase 2.} Otherwise, suppose that $X_s(e) = 0$ and $2e$ lies at the right end of a copy of $I_0$. Then we pursue the following condition: inside the output $\mathcal{F}$, $2e$ should ``move'' to the left end of a block $I_0$.
    
    In order to do this, we add one fresh number in-between $B$ and $C$, and $q$ fresh numbers in-between $C$ and $D$.
    
    In both subcases, we declare that the newly added odd numbers belong to the set of companions of $2e$.
    
    
    \emph{The construction} is arranged similarly to the previous ones.
    
    \emph{Verification.} 
    Since almost every block from $\alpha_f$ is isomorphic to $I_0$, every application of a PtR-module is successful. In addition, the actions of the PtR-module for $\mathcal{R}_e$ does not injure other strategies.
    
    
    We deduce that all requirements $\mathcal{R}_e$ are satisfied. Given $f_{\mathcal{A}}$ as an oracle, one can recover the position of $2e$ inside its $f_{\mathcal{A}}$-block. This implies that $X \leq_T f_{\mathcal{A}}$. In addition, a standard argument shows that $\mathcal{A}$ is a computable copy of $(\omega, <)$.
    
    Notice the following. Since $2e$ and its companions always stick together as an interval, there are only two possible variants of the final $f_{\mathcal{A}}$-block of $2e$: either it contains $q$ companions of $2e$ added at the very beginning of the work of the $\mathcal{R}_e$-strategy, or it contains $q$ closest (inside $\mathcal{A}$) companions of $2e$ added by the first application of the PtR-module for $\mathcal{R}_e$. 
    
    As in the previous case, given an odd number $x = 2k+1$, one can determine which of the following two cases holds:
    \begin{enumerate}
        \item $x$ is a companion of some even number $2e$ (the index $e$ is recovered effectively), or
        
        \item $x$ is added as a ``filler'' by some action of an $\mathcal{R}_e$-strategy.
    \end{enumerate}
    In the second case, the value $f_{\mathcal{A}_s}(x)$ never changes. In the first case, the oracle $X$ can tell us the content of the final $f_{\mathcal{A}}$-block containing $x$: indeed, if $X_{s_0}(e) = X_{s_1}(e)$, then at the stages $s_0$ and $s_1$, the blocks of $x$ inside $\mathcal{A}_{s_0}$ and $\mathcal{A}_{s_1}$ contain precisely the same elements. We deduce that $f_{\mathcal{A}} \leq_T X$. This concludes the proof of Case~(b).
    
    \medskip
    
    \textsc{Case~(c).} Before describing the construction, we provide a combinatorial analysis of the string $\alpha_f$.
    
    \begin{lemma}\label{lem:combin-analysis}
        If the string $\alpha_f$ satisfies neither Case~(a) nor Case~(b), then there are symbols $b,d,e \in \Sigma$ such that $d\neq b$, $e\neq b$, and for every natural number $n$, there exists $m > n$ such that the finite string $db^me$ occurs in $\alpha_f$.
    \end{lemma}
    \begin{proof}
        Without loss of generality, one may assume that every symbol from $\Sigma$  occurs infinitely often in $\alpha_f$.
    
        For a finite string $\sigma$ over the alphabet $\Sigma$, we denote
        \[  
            \#(\sigma) = | \{ a\in \Sigma\,\colon a \text{ occurs in } \sigma\} |.
        \]
        We choose a finite string $\tau$ such that $\tau$ occurs infinitely often in $\alpha_f$ and
        \begin{equation}\label{equ:combin-proof}
            \#(\tau) = \max\{ \#(\sigma)\,\colon \sigma \text{ occurs infinitely often in } \alpha_f\}.
        \end{equation}
        Let $c$ be the last symbol of the string $\tau$.
        
        There exists a symbol $b$ such that the string $\tau_b = \tau\,b$ occurs infinitely often in $\alpha_f$. Equation (\ref{equ:combin-proof}) implies that $b$ occurs in $\tau$ (indeed, if $b$ does not occur in $\tau$, then $\#(\tau_b) = \#(\tau) + 1$).
        
        We prove that $c = b$. Towards a contradiction, assume that $c\neq b$. Then $\tau$ can be decomposed as $\tau = \xi\, b\, \delta \, c^k$ for some $k\geq 1$ and finite strings $\xi,\delta$. The string $\tau_b = \xi\, b\, \delta \,c^k\,b$ occurs infinitely often in $\alpha_f$. In turn, this implies that both $b\,\delta\,c^k$ and $\delta\,c^k\,b$ occur infinitely often in $\alpha_f$. Therefore, $\alpha_f$ satisfies Case~(a), which gives a contradiction.
        
        Hence, we have $\tau = \rho\,b^k$ for some $k\geq 1$ and finite string $\rho$, and the string $\tau_b = \rho\,b^{k+1}$ occurs infinitely often in $\alpha_f$. Note that $\#(\tau_b) = \#(\tau)$. This implies that by applying induction, one can show that for \emph{every} $l\geq 1$, 
        \begin{equation}\label{equ:combin-proof-2}
            \rho\, b^l \text{ occurs infinitely often in } \alpha_f.
        \end{equation}
        
        Since $\alpha_f$ does not satisfy Case~(b), there are at least two different symbols occuring infinitely often in $\alpha_f$. This fact and (\ref{equ:combin-proof-2}) imply that for every $n\in\omega$, there exist $m > n$ and two symbols $d'$ and $e'$ such that $d'\neq b$, $e'\neq b$, and $d'b^m e'$ occurs in $\alpha_f$. After that,  we apply the pigeonhole principle to finish the proof of the lemma.
    \end{proof}
    
    By Lemma~\ref{lem:combin-analysis}, we may assume that for every $n\in\omega$, there exists $m > n$ such that, say, $10^m2$ occurs in $\alpha_f$. We satisfy the same requirements as in Case~(b):
    \[
        \begin{array}{l} 
            e\in X \ \Leftrightarrow\ 2e \text{ lies at the right end of a copy of } I_0,\\
            e\not\in X \ \Leftrightarrow\ 2e \text{ lies at the left end of a copy of } I_0.
        \end{array}
        \tag{$\mathcal{R}_e$}
    \]
    In general, our notations also follow those of Case~(b).
    
    \emph{Strategy $\mathcal{R}_e$ in isolation.}
    When $\mathcal{R}_e$ starts working at a stage $s_0+1$, we proceed as follows. Suppose $X_{s_0}(e) = 0$. We find a large enough number $m$ such that $10^m2$ occurs in $\alpha_f$, and the corresponding sequence of $f$-blocks $I_1 + I_0 + I_0 + \ldots + I_0 + I_2$ does not intersect with the image of $\mathcal{A}_{s_0}$ inside $(\omega,<)$. 
    
    We add $2e$ and fresh odd numbers into $\mathcal{A}$ ensuring that the newly added elements form a sequence of $f_{\mathcal{A}}$-blocks:
    \[
        I_1 + \underbrace{I_0 + \ldots + I_0}_{m \text{ times}} + I_2;
    \]
    if needed, fresh odd numbers are also added in-between $\mathcal{A}_{s_0}$ and this sequence. The number $2e$ lies at the left end of the leftmost block $I_0$. The elements forming $I_1$ and $I_2$ are declared \emph{boundary companions} of $2e$. The odd numbers forming the inner sequence of $I_0$-s are declared \emph{non-boundary companions} of $2e$.

    As usual, $\mathcal{R}_e$ \emph{requires attention} at a stage $s$ if inside the current $\mathcal{A}_s$, the corresponding requirement is not satisfied. When $\mathcal{R}_e$ requires attention, we apply a PtR-module.
    
    \emph{The PtR-module for $\mathcal{R}_e$.} 
    We choose the middle interval $C$ as the set containing $2e$ and all its companions. Assume that right now, $X_s(e) = 0$ and $2e$ lies at the right end of a copy of $I_0$ (the other subcase is treated in a similar way). Then the target condition is defined as follows: inside  $\mathcal{F}$, the number $2e$ belongs to the left end of a copy of $I_0$.
    
    Suppose that right now, the companions of $2e$ form a sequence of $f_{\mathcal{A}_s}$-blocks corresponding to a finite string $10^m2$. 
    
    We always assume the following: if a fresh number $x$ is added between some companions of some $2j$, then it is declared a non-boundary companion of $2j$. In addition, every such $x$ is put between the $I_1$-block and the $I_2$-block containing the boundary companions of $2j$. Moreover, we require that inside the resulting structure $\mathcal{F}$, the element $x$ becomes a part of a copy of $I_0$.
    
    In order to achieve the target condition, we proceed as follows. First, we find a large enough $m' > m$ such that $10^{m'}2$ occurs in $\alpha_f$, and this occurrence of $10^{m'}2$ lies to the right of the image of $\mathcal{A}_s$ inside $(\omega, <)$. We add fresh odd numbers in such a way that:
    \begin{itemize}
        \item The companions of $2e$ (including newly added companions) form a sequence of $f_{\mathcal{F}}$-blocks corresponding to $10^{m'}2$ (inside $\alpha_f$). This is achieved by adding numbers in-between $B$ and $C$, and by adding fresh $I_0$-blocks between the $I_1$-block and the $I_2$-block containing the boundary companions of $2e$.
        
        \item Similarly to Case~(b), this procedure must ensure that $2e$ moves to the left end of an $I_0$-block.
    \end{itemize}
    
     Second, we carefully push the companions of $2j$, where $e < j <s$, to the right. Consider each such $j$ (in turn).  Suppose that the companions of $2j$ form a sequence of $f_{\mathcal{A}_s}$-blocks corresponding to a finite string $10^{m_j}2$. 
     We choose a large enough $m'_j > m_j$ (again, with $10^{m'_j}2$ occuring in $\alpha_f$ to the right of the image of the current (preliminary) version of $\mathcal{F}$). We add fresh numbers in such a way that:
     \begin{itemize}
         \item The companions of $2j$ (including new ones) form a sequence of $f_{\mathcal{F}}$-blocks corresponding to $10^{m'_j}2$ inside $\alpha_f$.
         
         \item If $x$ is a new companion of $2j$, then it belongs to a new $I_0$-block which corresponds to one of the underlined zeros in the following decomposition:
         \[
            10^{m'_j} 2 = 1 0^{m_j} \underline{0} \underline{0} \ldots \underline{0} 2.
         \]
     \end{itemize}
     This careful pushing allows to ensure that the PtR-module does not injure strategies $\mathcal{R}_j$, for $j \neq e$. Indeed, after the pushing, the value $f_{\mathcal{A}}(2j)$ does not change.
     
    
    
    

    
    \emph{The construction} is arranged in a similar way as before.
    
    \emph{Verification.} The fact that $\alpha_f$ contains occurrences of $10^m2$ for arbitrarily large $m$ implies that every application of a PtR-module is successful.  We deduce that all requirements $\mathcal{R}_e$ are satisfied.
  The rest of the verification is similar to that of Case~(b).
    This concludes the proof of Theorem~\ref{theorem:block-functions}.
\end{proof}

\subsection{Quasi-block functions}\label{sec:qbf}
The notion of a quasi-block function is a generalization of the notion of a block function. Unlike blocks which are disjoint and follow each other, quasi-blocks are increasingly larger and they are initial segments of $\omega$.

\begin{definition}\label{def:quasi_block_function}

We say that $f: \omega \to \omega$ is a quasi-block function if there are arbitrarily long finite initial segments of $\omega$ closed under $f$. If $f$ is a quasi-block function but not a block function, we call $f$ a proper quasi-block function. 

\end{definition}

\begin{example}

Euler's function is a function $\varphi$ such that if $n>0$, then $\varphi(n)$ is the number of such $m \leq n$ that $m$ and $n$ are relatively prime. $\varphi$ is a proper quasi-block function. Since $\varphi$ has a computable non-decreasing lower bound $\lfloor\sqrt{\frac{n}{2}}\rfloor$ diverging to $\infty$, the spectrum of $\varphi$ is equal to the c.e. degrees by Theorem \ref{thm:bound}.

\end{example}

\begin{example}

The function $nd: \omega \to \omega$ assigning to each $n>0$ the number of its divisors is a proper quasi-block function. 

\end{example}

Below we describe a method used to show that the degree spectrum of a certain unary recursive function $f$ consists exactly of c.e. degrees.

\begin{description}
\item[Retrieving the Successor module (RS)] on $(\omega,<,f )$, for $f$ recursive, is a scheme of algorithm which, for any computable copy $\mathcal{A}$ of $(\omega,<)$ and an initial segment $\mathcal I_t$ of $\mathcal A$ satisfying some condition $\mathcal R$ (to be specified in a concrete implementation) computes, uniformly in $t$ and relative to $f_\mathcal{A}$, a longer initial segment $\mathcal I_{t+1}$ of $\mathcal{A}$ satisfying $\mathcal R$, which enables us to construct an increasing sequence of initial segments $\mathcal I_0 \subset \mathcal I_1 \subset \ldots$.

\end{description}
Suppose that there exists a concrete implementation of the RS-module for $(\omega,<,f)$.
 We wish to show that the degree spectrum of $f$ on $(\omega, <)$ consists of exactly c.e. degrees. To this aim, we want to show that $Succ_{\mathcal{A}}$ is Turing-reducible to $f_{\mathcal{A}}$. We also observe that the reduction in the other direction works. We conclude that $Succ_{\mathcal{A}} \equiv_T f_{\mathcal{A}}$, hence $DgSp(Succ)=DgSp(f)$, i.e. they consist of all c.e. degrees. This conclusion is based on Proposition \ref{theorem:DgSp_Succ}.

Suppose that the initial segment of $\omega$ up to $n$ (according to $<$) has already been determined, along with its isomorphic image $\mathcal I_t$ in $(\omega, \prec)$. In this description we adopt a convention that the isomorphic image of each number $i$ is $k_i$. Observe that for each number $i$ such that $k_i \prec k_n$ we know how to determine its successor in $(\omega, \prec)$. In an application of the RS-module, given $k_n$---the rightmost element of $\mathcal I_t$---we get some $k_m$ and $m$ such that $k_n \prec k_m$ and $[k_0;k_m]_\mathcal{A}$ satisfies $\mathcal R$. We know that in $\mathcal{A}$ there are exactly $m-n-1$ elements between $k_n$ and $k_m$. Since the ordering $\prec$ is recursive, we can check elements one by one until we determine what elements (and in what order) are between $k_n$ and $k_m$. Thus we extended the initial segment $\mathcal{I}_t$ of $\mathcal{A}$ to a larger initial segment $\mathcal{I}_{t+1}$ satisfying $\mathcal{R}$ and we are able to retrieve more values of the successor in this structure.

\begin{theorem}\label{theorem:non_quasi_block}
The spectrum of any unary total computable non-quasi-block function is equal to the c.e. degrees.

\end{theorem}

\begin{proof}

We show that the RS module can be used for $(\omega,<,f)$. Given a computable copy $\mathcal A$ of $(\omega,<)$, we set $\mathcal{I}_0$ as the image of some initial segment of $(\omega,<)$ such that for every $n$ outside $\mathcal{I}_0$ there is $m \leq n$ such that $f(m)> n$. The condition $\mathcal{R}$ states that there is $j \in \mathcal{I}_t$ such that $f(j)>n$. 
Then if we already know $\mathcal{I}_t$ and want to determine $\mathcal{I}_{t+1}$, we calculate both $f(j)$ and $f_{\mathcal{A}}(k_j)$ from the condition $\mathcal{R}$, obtaining some values of these functions $m$ and $k_m$, each of them somewhere behind $n$ and $k_n$ in their sequences.
\end{proof}

\begin{theorem}\label{thm:bound}
If $f$ is a recursive  proper quasi-block function with a computable non-decreasing lower bound diverging to $+\infty$, then its spectrum consists of exactly c.e. degrees.
\end{theorem}

\begin{proof}

We claim that there exist only finitely many quasi-blocks closed under both $f$ and $f^{-1}$. Observe that if there were infinitely many such quasi-blocks, then $f$ would be a block function.
Observe also that if $f$ is as above, then we are able to calculate how many times each of its values is assumed.

We utilise the RS module. The segment $\mathcal{I}_0$ is any initial segment such that none of its super-quasi-blocks is closed under both $f$ and $f^{-1}$. Assume we already have a segment $\mathcal{I}_t$ of $\mathcal{A}$ retrieved. $\mathcal{I}_t$ satisfies the condition $\mathcal{R}$ stating that it is a initial segment which is not closed under both $f$ and $f^{-1}$. 

We wish to algorithmically construct $\mathcal{I}_{t+1}$ satisfying the same condition $\mathcal{R}$. If there is $n \in \mathcal{I}_t$ such that $f(n)>\mathcal{I}_t$, we set $\mathcal{I}_{t+1}$ as the segment consisting of all elements up to $f(n)$. If not, then there must be $m \in \mathcal I_t$ such that for some $n>\mathcal{I}_t$, $f(n)=m$. What is more, for every such $m$ there are only finitely many arguments satisfying this identity and we are able to determine what they are. If $M$ is the largest of these elements, then we set $\mathcal{I}_{t+1}$ as the segment until $M$.
\end{proof}
\begin{restatable}{theorem}{theoremMichal}\label{thmMichal}
There exists a recursive quasi-block function $f$ with a non-decreasing lower bound diverging to $+\infty$ but with no such computable bound with all c.e. degrees as a spectrum.

\end{restatable}
\begin{proof}

Consider a set $A \subseteq \omega$ which is $\Delta_2$ but not computable. Observe that for each such set there is a recursive sequence $g$ of natural numbers such that each natural number appears in $g$ at most finitely many times and for any $n \in \omega$, $n \in A$ iff the number of occurrences of $n$ in $g$ is odd.

$f$ is going to be $g$ modified in such a way that we put some fixed points between elements of $g$, pushing these elements to the right, to ensure that $f$ is a quasi-block function. We will be able to easily distinguish (within $f$) old elements of $g$ from the new filler elements, because only the new elements are going to be fixed points of $f$.

We construct $f$ by finite extension, starting from the empty function. Initially, all elements of sequence $g$ are unused.
At any given stage, suppose that $g(m)$ is the least unused element of sequence $g$ and that $n$ is the least argument such that $f(n)$ is not defined yet. If $g(m)>n$, then for each $i=n, \ldots, g(m)$ assign $f(i)=i$. Regardless of whether you performed the previous instruction, assign value $g(m)$ to the least $i$ such that $f(i)$ has no value set yet. If the least such $i$ is equal to $g(m)$, then put an additional fixed point before it to ensure that all fixed points serve as fillers in $g$. We declare that $g(m)$ is used and go to the next stage.

This is a quasi-block function because each argument $n$ is either a fixed point or is a number from sequence $g$ which has been pushed so far to the right that $f(n)<n$. Hence every finite initial segment of $\omega$ is closed under $f$. However, this is not a block function. If it were, then every $m$ such that $f(m)=n$ would need to be in the same block as $n$. Then we would be able to count how many times $n$ is assumed as the value of $f$ and hence $A$ would be decidable. 

The lower bound of $f$ diverges to $\infty$ because every value can be assumed only finitely often. However, no such bound is computable because otherwise we would be able to determine the last occurrence of every number in $g$ and $A$ would be computable. Observe we can assume that this bound is non-decreasing. We just need to set $f(n)=$ the largest $m$ such that $f(i) \geq m$ whenever $i \geq n$.

If $A$ is a c.e. set, then we utilise the RS module to show that the degree spectrum of $f$ consists of exactly the c.e. degrees. We can assume without loss of generality that $g$ assumes each of its values only once, then so does $f$ if we ignore fixed points.

 We take $\mathcal{I}_0$ such that behind it there are no quasi-blocks closed under $f^{-1}$. The condition $\mathcal{R}$ states that there is an element $n>\mathcal{I}_t$ such that $f(n) \in \mathcal{I}_t$. Observe that such element is determined uniquely. We want to retrieve $\mathcal{I}_{t+1} \supseteq \mathcal{I}_t$ satisfying $\mathcal{R}$. We need to look for $n$ described above and then to fill in all the missing numbers between $\mathcal{I}_t$ and $n$. Since the segment thus obtained is not a block, it needs to satisfy $\mathcal{R}$. We call this segment $\mathcal{I}_{t+1}$.
\end{proof}

\subsection{Unusual degree spectrum}\label{sec:unusual}
In this section we answer Wright's question (Question 6.2 in \cite{wright_degrees_2018}). The result we prove here is also relevant for Harrison-Trainor's question (p. 5 in \cite{harrison-trainor_degree_2018}). Recall a representation of a block function $f$ as infinite sequence $\alpha_f$ of (the indices of) types (see, Remark \ref{representation}).
\begin{definition}
Let $f$ be a computable block function with infinitely many types. The counting function for $f$ is defined by $c_f(n) = \#\{i : \alpha_f(i) = n\}$.
\end{definition}

\begin{proposition}\label{thm:nikolay}
Let $f$ be a computable block function with infinitely many pairwise non-embeddable types, each occurring finitely often. Then $\deg(c_f)$ is c.e. and $f_\mathcal{A} \geq_T c_f$ implies that $ deg(f_\mathcal{A})$ is c.e.
\end{proposition}
\begin{proof}
$C_f^{\leq} := \{(k,n): k \leq c_f(n)\}$ is c.e., $C_f^{\geq} := \{(k,n): k \geq c_f(n)\}$ is co-c.e., so $\deg(C_f^{\leq} \oplus C_f^{\geq})$ is c.e. Since $C_f^{\leq} \oplus C_f^{\geq} \equiv_T c_f$, $c_f$ is of c.e. degree.

Assume that $f_\mathcal{A} \geq_T c_f$. Hence, $Succ_\mathcal{A} \leq_T f_\mathcal{A}$. But $Succ_\mathcal{A} \geq_T f_\mathcal{A}$ always (for a computable $f$). Hence $f_\mathcal{A} \equiv_T Succ_\mathcal{A}$, and thus, by Proposition \ref{theorem:DgSp_Succ}, $f_\mathcal{A}$ is of c.e. degree. 
\end{proof}

\begin{theorem}\label{theorem:unusual_spectrum}
There exists a total computable function whose degree spectrum strictly contains all c.e. degrees and is strictly contained in the $\Delta_2$ degrees.
\end{theorem}
We construct a computable block function $f$ with infinitely many types and each $c_f(n)$ finite. We want $c_f <_T \mathbf{0}'$ and a computable copy $\mathcal A$ of $(\omega,<)$ with $f_\mathcal{A}$ of non-c.e. degree. 
Combining this with Proposition\ref{thm:nikolay} and a result by Cooper, Lempp and Watson from \cite{cooper_weak_1989} (see Theorem \ref{thm:cooper_lempp_watson}) finishes the proof.

For each $e,e_1,e_2,n \in \omega$, we have the following requirements:
\begin{align*}
    & & \mathcal I_e: I  \not\simeq \Phi_e^J, & & \mathcal J_e: J \not\simeq \Phi_e^I, & &\text{and}&& R_{\langle e_1,e_2,n \rangle}: \Phi_{e_1}^{\Gamma_{f_\mathcal{A}}} \not\simeq W_{n} \vee \Phi_{e_2}^{W_n} \not\simeq \Gamma_{f_\mathcal{A}},
\end{align*}
where $\Gamma_{f_{\mathcal A}}$ is the graph of $f_{\mathcal{A}}$. The non-c.e. degree requirements are based on \cite[p. 195]{epstein_degrees_1979}.

At stage $s$ we have finite sets $I_s,J_s$, structure $\mathcal{A}_s = (A_s, <_{\mathcal{A}_s})$ and a function $f_{\mathcal{A}_s}: A_s \rightarrow A_s$. Eventually, we set $\mathcal{A} = \bigcup_{s\in \omega}\mathcal{A}_s$. We assume some recursive $\omega$-type ordering of $\mathcal I_e$, $\mathcal J_e$, $\mathcal R_{\langle e_1, e_2, n \rangle}$, for all $e, e_1, e_2, n \in \omega$. During construction, requirements reserve numbers and, in order to be satisfied, they wait until those numbers meet certain conditions, in which case we say that they need attention. 
\begin{itemize}
    \item $\mathcal I_e$ (or $\mathcal J_e$) needs attention at stage $s+1$, if some $x$ reserved for it at stage $s$ and $I_s(x) = \Phi_{e,s}^{J_s}$ (or $J_s(x) = \Phi_{e,s}^{I_s}$).
\item $R_{\langle e_1,e_2,n \rangle}$ needs attention at stage $s+1$ if, at stage $s$, some $\langle u,v \rangle$ is reserved for it, along with certain $t_0, t_1, t_2$ (called tickets), and, for some $z$, $\langle u,v \rangle < z < s$:
\begin{align*}
    & &(\alpha) \,\,\,\Phi_{e_1, s}^{\Gamma_{f_{\mathcal A_s}}}[z] = W_{n,s}[z]
    && \text{and}&& (\beta) \,\,\,\Phi_{e_2, s}^{W_{n,s}[z]}(\langle u, v \rangle) = \Gamma_{f_{{\mathcal{A}}_s}}(\langle u, v \rangle).
\end{align*}
 
\end{itemize}

We use a variant of PtR (the proof of Theorem~\ref{theo:finite-range}). In each application of PtR we distinguish $E$---the set of fresh numbers---for which we formulate an additional $E$-condition.  
\subsubsection{Construction}
Let $(\mathcal C_i)_{i \in \omega}$ be a computable sequence of cycles, where $\mathcal C_i$ is of length $2^i$ (Figure \ref{fig:cycle}).
\begin{figure}
    \begin{tikzpicture}[scale=0.6]
\node (a) at  (0,0) {0};
\node (c) at  (9.5,0) {$2^{i}-1$};
\node (a1) at (1.5,0) {1};
\node (a2) at (3,0) {2};
\draw[->] (a)  to (a1);
\draw[->] (a1) to (a2);
\node (a3) at (4.5,0) {};
\node (aaa) at (5.5,0) {};
\draw[->] (a2) to (a3);
\node (aaaa) at (5,0) {...};
\node (c1) at (7.3,0) {$2^{i}-2$};
\draw[->] (aaa) to (c1);
\draw[->] (c1) to (c);
\draw[->] (c) to [out=10,in=170] (a);
\end{tikzpicture}
    \caption{$\mathcal C_i = ([0;2^i-1], <, f_i)$, where the order $<$ is standard and $f_i$ corresponds to the arrows.}
    \label{fig:cycle}
\end{figure}
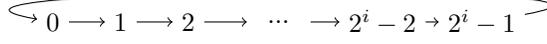
Put $I_0 = J_0 = \mathcal A_0 = f_{\mathcal{A}_0} = \emptyset$. Requirements have no reserved numbers, no numbers are frozen. Below we describe stage $s+1$, for $s \in \omega$.

\begin{enumerate}
    \item If no requirement needs attention at stage $s+1$, we choose the highest priority requirement with no reservation. If this is some $\mathcal I_e$ (or $\mathcal J_e$), we reserve for it the least fresh number $x$. If the highest priority requirement with no reservation is some $\mathcal R_{\langle e_1, e_2, n \rangle}$, we reserve for it the least number $\langle u,v\rangle $, fresh for $\mathcal A_s$ (i.e. $u,v$ do not occur in $\mathcal A_s$), and three \emph{consecutive} fresh numbers $t_0, t_1, t_2$, called tickets. We apply PtR by setting $B = \mathcal A_s$, $C=D =\emptyset$ and $E \supseteq \{u,v\}$ such that $|E| = 2^{t_0} + 2^{t_1}$ with every $x \in E$ being fresh for $\mathcal A_s$. We build a structure $\mathcal E =(E,<_\mathcal{E},g)$ where $<_{\mathcal{E}}$ is a linear order satisfying the $E$-condition, depicted in Figure \ref{fig:reservation}, which is:
\begin{itemize}
    \item $\mathcal C_{t_0} + \mathcal C_{t_1} \cong \mathcal E$, 
    \item $u$ is the $<_\mathcal{E}$-last element in the block corresponding to $\mathcal C_{t_0}$, and
    \item  $v$ the $<_\mathcal{E}$-first element in the block corresponding to $\mathcal C_{t_1}$.
\end{itemize} 
We set $\mathcal A_{s+1} = \mathcal A_{s} + \mathcal E$. We have $\langle u,v \rangle \notin \Gamma_{f_{\mathcal{A}_{s+1}}}$. We enumerate ticket $t_0$ into $I$. 

\begin{figure}
    \centering
    \begin{tikzpicture}
\draw[densely dashed] (0,0) -- (2,0) node[style={midway},above] {$\mathcal{A}_s$};
\node[line width=0.5mm,gray, draw, rounded rectangle, label=$\mathcal C_{t_0}$] at (2.65,0) {\hspace{0.7cm}$^u\cdot$};
\node[line width=0.5mm, gray, draw, rounded rectangle, label=$\mathcal C_{t_1}$] at (4.28,0) {$\cdot^v\hspace{1.4cm}$};
\end{tikzpicture}
    \caption{$\mathcal A_{s+1}$ after reserving $\langle u,v\rangle$ and tickets $t_0,t_1,t_2$ for $\mathcal{R}_{\langle e_1, e_2, n \rangle}$. }
    \label{fig:reservation}
\end{figure}
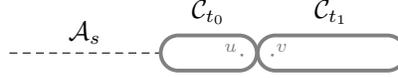
\item If a requirement needs attention, pick the highest one. We say it receives attention. If this is $\mathcal I_e$, some $x$ is reserved for $\mathcal I_e$ at stage $s$ and $I_s(x) = \Phi_{e,s}^{J_s}(x)$. Put $x$ into $I$, freeze the computation $\Phi^{J_s}_{e,s}(x)$ and cancel all freezings and reservations for lower priority requirements. Deal with with $\mathcal J_e$ accordingly.

Suppose the highest priority requirement needing attention is some $\mathcal R_{\langle e_1, e_2, n \rangle}$. Some $\langle u, v \rangle$ is reserved for $\mathcal R_{\langle e_1, e_2, n \rangle}$ at stage $s$ with some tickets $t_0,t_1,t_2$. Below we describe reactions to first and second attention received by $\mathcal R_{\langle e_1, e_2, n \rangle}$ with reservation $\langle u,v \rangle, t_0,t_1,t_2$. 
\begin{enumerate}
    \item[(i)] Suppose the reservation for $\mathcal R_{\langle e_1, e_2, n \rangle}$ has been made at stage $r$. After $r$ and before $s+1$ the structure $\mathcal A$ might have been extended by some $\mathcal T$ (thick line in Figure \ref{fig:1stattentionBefore}).  
\begin{figure}
    \centering
    \begin{tikzpicture}
\draw[densely dashed] (0,0) -- (2,0) node[style={midway},above] {$\mathcal{A}_{r-1}$};
\node [line width=0.5mm, draw, rounded rectangle, label=$\mathcal C_{t_0}$] at (2.65,0) {\hspace{0.7cm}$^u\cdot$};

\node [line width=0.5mm, draw, rounded rectangle, label=$\mathcal C_{t_1}$] at (4.3,0) {$\cdot^v\hspace{1.4cm}$};
\draw[line width=0.5mm] (5.3,0) -- (7.2,0) node[style={near end},above] {\footnotesize{\hspace{1cm} $\mathcal T$ (added at stages $> r$)}};
\end{tikzpicture}
    \caption{$\mathcal A_{s}$ when $R_{\langle e_1, e_2, n \rangle}$ receives attention for the first time with $\langle u,v \rangle$ and tickets $t_0,t_1,t_2$, assuming that the reservation has been made at stage $r$.}
    \label{fig:1stattentionBefore}
\end{figure}
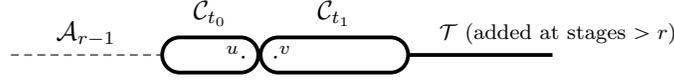
The idea is that we push to the right all numbers that occupy the highlighted positions in Figure \ref{fig:1stattentionBefore} and obtain the structure as in Figure \ref{fig:1stattentionAfter}. 

More formally, divide $\mathcal{A}_{s}$ into $\mathcal{A}_s = \mathcal B + \mathcal C + \mathcal D$, where $\mathcal D \cong \mathcal A_{r-1}$,  $\mathcal C \cong \mathcal C_{t_0} + \mathcal C_{t_1}$ and $\mathcal D \cong \mathcal T$, and apply PtR. Take $|C \cup D|$ numbers, fresh for $\mathcal A_s$, and make $F$ out of them. Build a structure $\mathcal F = (F, <_\mathcal{F}; g)$, where $<_\mathcal{F}$ is a linear order, satisfying the $F$-condition $\mathcal F \cong \mathcal C + \mathcal D$. We rebuild $\mathcal C$ to get $\mathcal C' = (C, <_\mathcal{C}; h)$ where $\mathcal C'$ satisfies the $C$-condition $\mathcal C' \cong \mathcal C_{t_1} + \mathcal C_{t_0}$. We set $\mathcal A_{s+1} = \mathcal B + \mathcal F + \mathcal C' + \mathcal D$ (Figure \ref{fig:1stattentionAfter}).
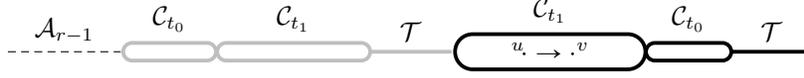
\begin{figure}
    \centering
    \begin{tikzpicture}
\draw[densely dashed] (0,0) -- (1.5,0) node[style={midway},above] {$\mathcal{A}_{r-1}$};
\node [line width=0.5mm, draw, rounded rectangle, label=$\mathcal C_{t_0}$, lightgray] at (2.15,0) {$\hspace{1cm}$};
\node [line width=0.5mm, draw, rounded rectangle, label=$\mathcal C_{t_1}$, lightgray] at (3.8,0) {$\hspace{1.8cm}$};
\draw[line width=0.5mm, lightgray] (4.8,0) -- (5.9,0) node[style={midway},above,black] {$\mathcal{T}$};
\node [line width=0.5mm, draw, rounded rectangle, label=$\mathcal C_{t_1}$] at (7.21,0) {$\hspace{0.6cm}{}^u\! \cdot\to\cdot^v\hspace{0.6cm}$};
\node [line width=0.5mm, draw, rounded rectangle, label=$\mathcal C_{t_0}$] at (9.05,0) {$\hspace{0.9cm}$};
\draw[line width=0.5mm] (9.6,0) -- (10.7,0) node[style={midway},above,black]{$\mathcal T$};
\end{tikzpicture}
    \caption{The result of reaction to \emph{first attention} for $\mathcal R_{\langle e_1,e_2,n \rangle}$ with reservation $\langle u,v\rangle$ and tickets $t_0,t_1,t_2$. Gray part is occupied by fresh numbers, thick part represents pushed numbers.}
    \label{fig:1stattentionAfter}
\end{figure}

Observe that pushed numbers from $\mathcal C + \mathcal D$ assume in $\mathcal A_{s+1}$ the same order structure as in $\mathcal A_{s}$ but the behavior of $f_{\mathcal{A}_{s+1}}$ mimics that on $\mathcal C_{t_1} + \mathcal C_{t_0}$. This makes $\Gamma_{f_{\mathcal A_{s+1}}}(\langle u, v \rangle) = 1$ and thus $\mathcal R_{\langle e_1,e_2,n \rangle}$ is satisfied at stage $s+1$. We enumerate $t_1$ into $I$ and invalidate all reservations and freezings for lower priority requirements.
\item[(ii)]  Suppose $\mathcal R_{\langle e_1,e_2,n \rangle}$ has made the reservation at stage $r$ and received the first attention at stage $p+1$. By the time we got to stage $s+1$, the structure $\mathcal A$ might have been extended by some $\mathcal U$  (Figure \ref{fig:2ndattentionBefore}). 
\begin{figure}
    \centering
    \begin{tikzpicture}
\draw[densely dashed] (0,0) -- (0.7,0) node[style={midway},above] {$\mathcal{A}_{r-1}$};
\node [draw, rounded rectangle, label=$\mathcal C_{t_0}$] at (1,0) {$\hspace{0.4cm}$};
\node [draw, rounded rectangle, label=$\mathcal C_{t_1}$] at (1.85,0) {$\hspace{0.8cm}$};
\draw[] (2.36,0) -- (3.1,0) node[style={midway},above] {$\mathcal T$};
\node [line width=0.5mm,draw, rounded rectangle, label=$\mathcal C_{t_1}$] at (3.8,0) {${}^u\! \cdot\to\cdot^v$};
\node [line width=0.5mm,draw, rounded rectangle, label=$\mathcal C_{t_0}$] at (4.8,0) {$\hspace{0.4cm}$};
\draw[line width=0.5mm] (5.1,0) -- (5.8,0) node[style={midway},above] {$\mathcal T$};
\draw[dotted] (5.85,-0.5) -- (5.85,0.5);
\draw[line width=0.5mm, densely dotted] (5.8,0) -- (6.5,0) node[style={near end},above] {\hspace{3.5cm} $\mathcal U$ added at stages $> p+1$};;
\end{tikzpicture}
    \caption{$\mathcal A_{s}$ when $R_{\langle e_1, e_2, n \rangle}$ receives attention for the second time with $\langle u,v \rangle$ and tickets $t_0,t_1,t_2$, assuming that the reservation has been made at stage $r$.}
    \label{fig:2ndattentionBefore}
\end{figure}
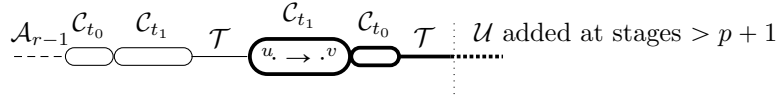
The idea is that we push all numbers occupying the highlighted positions in Figure \ref{fig:2ndattentionBefore} and obtain the structure as in Figure \ref{fig:2ndattentionAfter}.
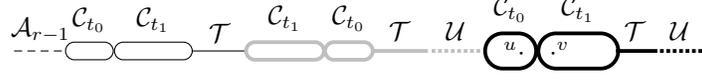
\begin{figure}
    \centering
    \begin{tikzpicture}
\draw[line width=0.5mm, lightgray] (4.75,0) -- (5.5,0)node[style={midway},above,black]{$\mathcal T$};
\draw[line width=0.5mm, densely dotted, lightgray] (5.55,0) -- (6.2,0) node[style={midway},above,black]{$\mathcal U$};
\draw[densely dashed] (0,0) -- (0.7,0) node[style={midway},above] {$\mathcal{A}_{r-1}$};
\node [draw, rounded rectangle, label=$\mathcal C_{t_0}$] at (1,0) {$\hspace{0.4cm}$};
\node [draw, rounded rectangle, label=$\mathcal C_{t_1}$] at (1.85,0) {$\hspace{0.8cm}$};
\draw[] (2.36,0) -- (3.1,0) node[style={midway},above] {$\mathcal T$};
\node [line width=0.5mm,draw, rounded rectangle, label=$\mathcal C_{t_1}$, lightgray] at (3.6,0) {$\hspace{0.8cm}$};
\node [line width=0.5mm,draw, rounded rectangle, label=$\mathcal C_{t_0}$, lightgray] at (4.45,0) {$\hspace{0.4cm}$};
\node [line width=0.5mm,draw, rounded rectangle, label=$\mathcal C_{t_0}$] at (6.6,0) {$\hspace{0.1cm}^u\cdot$};
\node [line width=0.5mm,draw, rounded rectangle, label=$\mathcal C_{t_1}$] at (7.5,0) {$\cdot^v\hspace{0.5cm}$};
\draw[line width=0.5mm] (8,0) -- (8.5,0) node[style={midway},above]{$\mathcal T$};
\draw[line width=0.5mm, densely dotted] (8.5,0) -- (9.2,0) node[style={midway},above]{$\mathcal U$};
\end{tikzpicture}
    \caption{The result of reaction to \emph{second attention} of $\mathcal R_{\langle e_1,e_2,n \rangle}$. Gray part is occupied by fresh numbers, thick part represents pushed numbers.}
    \label{fig:2ndattentionAfter}
\end{figure}

More formally, we divide $\mathcal A_s = \mathcal B + \mathcal C + \mathcal D$ in a way that $\mathcal B \cong \mathcal A_{r-1} + \mathcal C_{t_0} + \mathcal C_{t_1} + \mathcal T$,  $\mathcal C \cong \mathcal C_{t_1} + \mathcal C_{t_0}$ and $\mathcal D \cong \mathcal T + \mathcal U$ with $u,v$ residing in a copy of $\mathcal C_{t_1}$ within $C$. We apply PtR with $\mathcal B, \mathcal C, \mathcal D$ defined above. Let $F$ be the set of $|C \cup D|$ numbers, fresh for $\mathcal A_s$. We build a finite structure $\mathcal F = (F,<_\mathcal{F}, g)$, where $<_\mathcal{F}$ is a linear order, satisfying the $F$-condition $\mathcal F \cong \mathcal C + \mathcal D$. We rebuild $\mathcal C$ to get $\mathcal C' = (C, <_\mathcal{C}, h)$ satisfying the $C$-condition $\mathcal C' \cong C_{t_0} + C_{t_1}$.
We set $\mathcal{A}_{s+1} = \mathcal B + \mathcal F + \mathcal C' + \mathcal D.$
We have $\Gamma_{f_{\mathcal{A}_{s+1}}}(\langle u,v \rangle) = 0$. $\mathcal R_{\langle e_1,e_2,n \rangle}$ is satisfied at stage $s+1$. We enumerate $t_2$ into $I$ and invalidate all reservations and freezings for lower priority requirements.
\end{enumerate}
\end{enumerate}

\subsubsection{Verification}
\begin{restatable}{lemma}{lemmaAcomp}\label{lemmaAcomp}
$\mathcal A$ is computable.
\end{restatable}
\begin{proof}
It is clear that the domain of $\mathcal A$ is $\omega$. By the construction, once two numbers enter $\mathcal A$, their order according to $<_\mathcal{A}$ is never changed. Hence, $\mathcal A_s \subset \mathcal A_{s+1}$, for every $s$. Thus we can set $\mathcal A = \bigcup_{s\in\omega}\mathcal A_s$. Clearly, $<_\mathcal{A}$ is computable: to decide whether $x <_\mathcal{A} y$ holds it suffices to carry out the construction until some stage $s_0$ at which $x,y \in \mathcal{A}_{s_0}$. We know that $x <_\mathcal{A} y \iff x <_{\mathcal{A}_{s_0}}y$. 
\end{proof}

\begin{lemma}\label{lemma:satisfied}
Every requirement is eventually satisfied. Hence, $I,J$ are intermediate and $f_\mathcal{A}$ is of non-c.e. degree.
\end{lemma}
\begin{proof}
This follows from finite-injury. It remains to observe that each requirement can receive attention only finitely many times with the same numbers reserved for it. This is clear for $\mathcal I_e, \mathcal J_e$ (see, e.g. \cite[Chap.~VII.2]{soare_recursively_1987}). 
We show that no $\mathcal R_{\langle e_1,e_2,n \rangle}$ needs attention more than twice with the same $\langle u,v\rangle$ and tickets $t_0,t_1,t_2$ reserved for it (cf. \cite[p. 196]{epstein_degrees_1979}). Suppose the reservation was made at stage $r$, the first attention was at stage $s+1$ and the second at stage $t+1$. Since $\langle u,v\rangle, t_0,t_1,t_2$ are reserved for $\mathcal R_{\langle e_1,e_2,n \rangle}$ at stage $t \geq s+1$, no requirement with lower priority than $\mathcal R_{\langle e_1,e_2,n \rangle}$ has received attention at any stage $u$, $t\geq u \geq s+1$.
Actions performed at stage $t+1$ lead to $\mathcal A_{t+1}\restriction A_{s} = \mathcal A_{s} \restriction A_{s}$. Therefore, $\Phi_{e_1}^{\Gamma_{f_{\mathcal{A}_{t+1}\restriction A_{s}}}}[z] = \Phi_{e_1}^{\Gamma_{f_{\mathcal{A}_{s}\restriction A_{s}}}}[z] =\Phi_{e_1}^{\Gamma_{f_{\mathcal{A}_{s}}}}[z] = W_{n,s}[z]$. At stage $s+1$ we had $\Phi_{e_2}^{W_{n,s}[z]} (\langle u, v \rangle)= \Gamma_{f_{\mathcal A_{s}}}(\langle u, v \rangle) \neq \Gamma_{f_{\mathcal A_{t}}}(\langle u, v \rangle)$. Since at stage $t+1$ we had $\Phi_{e_2}^{W_{n,t}[z]}(\langle u,v\rangle) = \Gamma_{f_{\mathcal A_{t}}}(\langle u, v \rangle)$ we must have $W_{n,t}[z] \neq W_{n,s}[z]$. Hence, for some $x$, $\Phi_{e_1}^{\Gamma_{f_{\mathcal{A}_{t+1}\restriction A_{s}}}}(x) = W_{n,s}(x) \neq W_{n,t}(x)$. Now, observe that  $\mathcal A_{t+1}\restriction A_{s}$ does not change at any later stage at which $\langle u,v\rangle$ is reserved for $\mathcal R_{\langle e_1,e_2,n\rangle}$. Hence, for all such stages $w \geq t+1$, $\Phi_{e_1,w}^{\Gamma_{f_{\mathcal{A}_{w}\restriction A_{s}}}}(w) \neq W_{n,w}(x)$ and $\mathcal R_{\langle e_1,e_2,n\rangle}$ does not need attention at stage $w+1$.
\end{proof}

\begin{lemma}\label{main_verification_lemma}
For every $n \in \omega$, $c_f(n)$ is never increased due to numbers $> n+2$ entering $I$.
\end{lemma}
\begin{proof}
Suppose the contrary. Then there exists $n$ such that $c_f(n)$ is increased because of some $k > n+2$ entering $I$. Let $s+1$ be the stage at which this happens. Since $c_f(n)$ is increased at stage $s+1$, $\mathcal C_{n}$ is present in $\mathcal A_{s+1}$. Since $c_f(n)$ is increased due to $k$ entering $I$, $k$ must be associated at stage $s+1$ with some $R_i$. Hence, $k$ is one of the tickets $t_0, t_1,t_2$ paired with $R_i$ at this point. There are three cases.
\begin{itemize}
\item[($k=t_0$)] This is when $R_i$ is initialized and receives tickets $t_0,t_1,t_2$ (see Figure \ref{fig:reservation}). For $c_f(n)$ to increase, we must have $n=t_0$ or $n = t_1$. $n=t_0$ is not possible because then we would have $k=t_0=n$ which contradicts $k>n+2$. $n=t_1$ is also not possible because we would have $k=t_0 = t_1 - 1 = n-1$ which contradicts $k > n+2$.
\item[($k = t_1$)] This is when $R_i$ receives first attention with tickets $t_0,t_1,t_2$ (see Figure \ref{fig:1stattentionBefore}). $\mathcal C_n$ must occur somewhere at the highlighted positions in Figure \ref{fig:1stattentionBefore} because this fragment of the structure is copied leading to an increase of $c_f$. Hence, $n=t_0$ or $n=t_1$, or $\mathcal C_{n}$ occurs in $\mathcal T$. $n \neq t_0$ because otherwise $n=t_0$, $k=t_1=t_0+1=n+1$ which contradicts $k > n+2$. $n$ cannot be $t_1$ because otherwise $n=t_1=k$ which contradicts $k> n+2$. Hence, $\mathcal C_n$ occurs in $\mathcal T$. However, this is also not possible for the following reason. We know that $k=t_1$ enters $I$ so this is due to $R_i$ acting when receiving the fist attention with tickets $t_0,t_1,t_2$. This means that no higher than $R_i$ requirement $R_j$ (i.e., with $j<i$) has received attention after $R_i$ got associated with tickets $t_0,t_1,t_2$ (up to the current stage)---otherwise $R_i$'s tickets would have been reassigned to numbers different than $t_0,t_1,t_2$. This means that $\mathcal C_n$ entered the construction \emph{after} $R_i$ was assigned to $t_0,t_1,t_2$. Therefore, by the construction (i.e. the way we choose and assign tickets to requirements (re)entering the construction), $n$ is a ticket for some lower priority requirement $R_l$ ($l > i$). But when $n$ enters the construction as a ticket of such $R_l$, $n$ is chosen as a fresh number so, in particular, $n > t_1 = k$ which contradicts $k>n+2$.
\item[($k=t_2$)] This is when $\mathcal R_i$ receives attention for the second time with tickets $t_0,t_1,t_2$ (see Figure \ref{fig:2ndattentionBefore}). $\mathcal C_n$ occurs somewhere at the highlighted positions in Figure \ref{fig:2ndattentionBefore}, i.e. $n=t_0$ or $n=t_1$, or $\mathcal C_n$ occurs in $\mathcal T + \mathcal U$. $n \neq t_0$ because otherwise $n=t_0$, $k=t_2=t_0+2 = n+2$ which contradicts $k > n+2$. $n \neq t_1$ because otherwise $n=t_1$, $k=t_2=t_1+1=n+1$ which contradicts $k > n+2$. Therefore, $n$ occurs in $\mathcal T + \mathcal U$. The rest of the argument is similar to the analogical place of the case previous case ($k=t_1$).
\end{itemize}
\end{proof}

\begin{lemma}\label{counting_function_lemma}
$c_f \leq_T I$.
\end{lemma}
\begin{proof} To compute $c_f(n)$, find $s$ such that $I_s[n+2] = I[n+2]$. By Lemma \ref{main_verification_lemma} and the fact that $c_f(n)$ is increased \emph{only} due to numbers entering $I$, $c_f(n)$ is not increased at stages $>s$ (no additional copies of $\mathcal C_n$ are added to $f_\mathcal{A}$). Return the number of copies of $\mathcal C_n$ in $f_{\mathcal A_s}$.
\end{proof}
\begin{restatable}{lemma}{lemmafc}\label{fcopy_lemma}
$f_\mathcal{A} \leq_T c_f$.
\end{restatable}
\begin{proof} 
To compute in $c_f$ the value $f_\mathcal{A}(x)$, wait for the earliest $s$ such that $x \in \mathcal A_s$. At stage $s-1$, either some requirement needed attention or some requirement made a reservation.

Assume that some requirement needed attention at stage $s-1$. Hence, $x$ is one of the fresh numbers that have been added to $\mathcal A$ in response the requirement that received attention at stage $s$. These numbers are represented by the gray fragments in Figures \ref{fig:1stattentionAfter}, \ref{fig:2ndattentionAfter}. It is crucial to observe that blocks created in this manner always preserve their structure throughout the construction. It means that when fresh numbers are put in such blocks, they can be pushed to the right (with no in-between insertions) but once they are pushed, they land on positions on which we recreate the behavior of $f$ from their original positions. Therefore, $f_\mathcal{A}(x) = f_{\mathcal{A}_s}(x)$.

Now, assume that some some requirement $\mathcal R_i$ made a reservation with $\langle u,v \rangle$ and tickets $t_0,t_1,t_2$ at stage $s-1$. Hence, $x$ is one of the fresh numbers that have been added to $\mathcal A$ at stage $s$ as elements of the cycles $\mathcal C_{t_0}$ and $\mathcal C_{t_1}$. These numbers are represented by the gray fragment in Figure \ref{fig:reservation}. Note that, later in the construction, $f$ can assume a different behavior on them in response to first or second attention received by $\mathcal R_i$ with reservation $\langle u, v \rangle$, $t_0,t_1,t_2$. However, when we take action in such a situation, $c_f(t_0)$ and $c_f(t_1)$ is increased. Since, by the construction, the behavior of $f$ on $x$ cannot change due to other reason than $\mathcal R_i$ receiving attention with reservation $\langle u,v \rangle$, $t_0,t_1,t_2$ ($x$ can be pushed to the right in other circumstances but then the behavior of $f$ on $x$ is faithfully reproduced). Therefore, to compute $f_\mathcal{A}(x)$ it suffices to ask the oracle for $c_f(t_0)$ (or $c_f(t_1)$) and carry out the construction up to a stage $s_0$ such that $\mathcal A_{s_0}$ contains $c_f(t_0)$ occurrences of cycle $\mathcal C_{t_0}$. Then we are sure that $f_\mathcal{A}(x) = f_{\mathcal{A}_{s_0}}(x)$.
\end{proof}

By Lemmas \ref{lemma:satisfied}, \ref{main_verification_lemma}, \ref{counting_function_lemma} and \ref{fcopy_lemma}: $\mathbf{0} <_T f_\mathcal{A} \leq c_f \leq_T I <_T \mathbf{0}'$. The spectrum of $f$ is not trivial by Proposition \ref{theorem:trivial}. By Theorem \ref{thm:wright}, $DgSp(f)$ contains all c.e. degrees. Since $f_A$ is of non-c.e. degree, $DgSp(f) \neq$ the c.e. degrees. To show that $DgSp(f)\neq$ the $\Delta_2$ degrees, we need
\begin{theorem}[Cooper, Lempp and Watson, \cite{cooper_weak_1989}]\label{thm:cooper_lempp_watson}
Given c.e. sets $U <_T V$ there is a proper d.c.e. set $C$ of properly d.c.e. degree such that $U <_T C <_T V$. 
\end{theorem}
Assume, for a contradiction, that $DgSp(f)$ consists of the $\Delta_2$ degrees. By Theorem \ref{thm:cooper_lempp_watson}, $DgSp(f) \cap \{\deg(A): c_f \leq_T A \leq_T \mathbf{0}'\}$ contains a properly d.c.e. degree. However, by Proposition \ref{thm:nikolay}, $DgSp(f) \cap \{\deg(A): c_f \leq_T A \leq_T \mathbf{0}'\}$ contains only c.e. degrees. This is a contradiction, so the degree spectrum of $f$ is different then the $\Delta_2$ degrees. This completes the proof.

\section{Conclusions and open questions}
 We have investigated the problem of intrinsic complexity of computable relations on $(\omega,<)$, as measured by their degree spectra, in the restricted setting of graphs of unary total computable functions. It has been known that possible candidates for intrinsic complexities of such functions include three sets consisting of precisely: the computable degree, all c.e. degrees, and all $\Delta_2$ degrees. Imposing certain structural contraints on such functions has led us to the notions of block functions (Definition \ref{def:block_function}) and a broader class of quasi-block functions (Definition \ref{def:quasi_block_function}). Non-quasi-block functions have intrinsic complexity equal to the c.e. degrees (Theorem \ref{theorem:non_quasi_block}) which redirects all focus to quasi-block functions. We have obtained several results on this class, most prominently the one on block-functions with finitely many types (Theorem \ref{theorem:block-functions}) showing that their intrinsic complexity is either trivial of equal to the $\Delta_2$ degrees. However, the most surprising result is that on an unusual degree spectrum (Theorem \ref{theorem:unusual_spectrum}) which proves the existence of a block function having intrinsic complexity different from the already known three candidates. To the best of our knowledge, this theorem answers Question 6.2 from \cite{wright_degrees_2018} formulated by Wright who asked whether there are relations on $(\omega,<)$ with other degree spectra (than the three known candidates). Harrison-Trainor obtained a related result though for a different relation. However, for his relation it is not known whether its spectrum is intermediate (see Section \ref{sec:introduction} for details, as well as \cite{harrison-trainor_degree_2018}). 

A few questions arise immediately. Although we have been able to obtain some results on computable block functions with infinitely many types, the spectrum problem for such functions remains largely unsolved. Even for the function constructed in Theorem \ref{theorem:unusual_spectrum}, the exact contents of its spectrum are unknown. We finish the paper with an open question: are there infinitely many spectra of unary total computable functions on $(\omega,<)$?
\bibliographystyle{plain}
\bibliography{ref}
\end{document}